\documentclass[11pt,a4paper]{amsart}
\usepackage{amssymb,amsmath,epsfig,graphics,mathrsfs}

\usepackage{fancyhdr}
\fancyheadoffset{0cm}
\usepackage{hyperref}
\hypersetup{
 colorlinks   = true,
 urlcolor     = blue,
 linkcolor    = blue,
 citecolor   = red ,
 bookmarksopen=true
}

\usepackage{amsmath}
\usepackage{amsfonts}
\usepackage{amssymb}
\usepackage{amsthm}
\usepackage{epsfig,graphics,mathrsfs}
\usepackage{graphicx}

\usepackage[usenames, dvipsnames]{color} 

\usepackage{hyperref}

\def \de {\partial}

\def \phi {\varphi}
\def \RNu {\mathbb{R}^{N+1}}
\def \RN {\mathbb{R}^N}
\def \R {\mathbb{R}}

\def \K {\mathscr{K}}

\def \G{\Gamma}
\newcommand{\Ba}{\mathscr B_z^{(a)}}
\newcommand{\paa}{z^a \de_z}
\def \vf{\varphi}
\def \S {\mathscr{S}(\R^{N+1})}
\def \So {\mathscr{S}}

\newcommand{\Rn}{\mathbb R^n}

\newcommand{\p}{\partial}

\newcommand{\la}{\lambda}

\numberwithin{equation}{section}

\newcommand{\beq}{\begin{equation}}
\newcommand{\bea}[1]{\begin{array}{#1} }
\newcommand{\eeq}{ \end{equation}}
\newcommand{\ea}{ \end{array}}

\newcommand{\Rnp}{\mathbb R^{N+1}_+}

\newcommand{\As}{(-\mathscr A)^s}
\newcommand{\sA}{\mathscr A}

\newcommand{\Po}{\mathscr P}

\newcommand{\Lp}{L^p}

\newcommand{\Li}{L^\infty}
\newcommand{\Lii}{L^\infty_0}

\newcommand{\rpp}{\rho_p(\sA)}

\def \tr{\mathrm{tr}}

\newtheorem{theorem}{Theorem}[section]
\newtheorem{lemma}[theorem]{Lemma}
\newtheorem{proposition}[theorem]{Proposition}
\newtheorem{corollary}[theorem]{Corollary}
\newtheorem{remark}[theorem]{Remark}
\newtheorem{definition}[theorem]{Definition}
\theoremstyle{definition}

\numberwithin{equation}{section}

\begin{document}

\title[A class of nonlocal hypoelliptic operators, etc.]{A class of nonlocal hypoelliptic operators\\ and their extensions}


\date{}

\begin{abstract} 
In this paper we study nonlocal equations driven by the fractional powers of hypoelliptic operators in the form
\[
\mathscr K u = \mathscr A u  - \de_t u \overset{def}{=} \operatorname{tr}(Q \nabla^2 u) + <BX,\nabla u> - \de_t u,
\] introduced by H\"ormander in his 1967 hypoellipticity paper. We show that the nonlocal operators $(-\mathscr K)^s$ and $\As$ can be realized as the Dirichlet-to-Neumann map of doubly-degenerate extension problems. We solve such problems in $\Li$, and in $\Lp$ for $1\le p<\infty$ when $\operatorname{tr}(B)\ge 0$. In  forthcoming works we use such calculus to establish some new Sobolev and isoperimetric inequalities. 
\end{abstract}

\author{Nicola Garofalo}

\address{Dipartimento d'Ingegneria Civile e Ambientale (DICEA)\\ Universit\`a di Padova\\ Via Marzolo, 9 - 35131 Padova,  Italy}
\vskip 0.2in
\email{nicola.garofalo@unipd.it}

\thanks{The first author was supported in part by a Progetto SID (Investimento Strategico di Dipartimento) ``Non-local operators in geometry and in free boundary problems, and their connection with the applied sciences", University of Padova, 2017.}

\author{Giulio Tralli}
\address{Dipartimento d'Ingegneria Civile e Ambientale (DICEA)\\ Universit\`a di Padova\\ Via Marzolo, 9 - 35131 Padova,  Italy}
\vskip 0.2in
\email{giulio.tralli@unipd.it}

\maketitle

\tableofcontents

\section{Introduction}\label{intro}

In 1967 H\"ormander proved his celebrated theorem stating that if for smooth vector fields $Y_0, Y_1,...,Y_m$ in $\R^{N+1}$ the Lie algebra generated by them has maximum rank, then the second order partial differential operator $\mathscr L = \sum_{i=1}^m Y_i^2 + Y_0$ is hypoelliptic, see \cite{Ho}. As a motivation to his study, in the opening of his paper the author considered the following class of equations
\begin{equation}\label{K0}
\mathscr K u = \mathscr A u  - \de_t u \overset{def}{=} \operatorname{tr}(Q \nabla^2 u) + <BX,\nabla u> - \de_t u = 0,
\end{equation}
and showed that $\K$ is hypoelliptic if and only if the covariance matrix
\begin{equation}\label{Kt}
K(t) = \frac 1t \int_0^t e^{sB} Q e^{s B^\star} ds
\end{equation} 
is invertible, i.e., $\operatorname{det} K(t) >0$ for every $t>0$. We note that the strict positivity of $K(t)$ is equivalent to the finite rank condition on the Lie algebra. 
In \eqref{K0} $Q$ and $B$ are $N\times N$ matrices with real, constant coefficients, with $Q\ge 0$, $Q = Q^\star$. We have denoted by $X$ the variable in $\R^N$, and thus $(X,t)\in \R^{N+1}$, and by $A^\star$ the transpose of a matrix $A$.

The class of operators \eqref{K0} includes several examples of interest in analysis, physics and the applied sciences. The simplest one is of course the ubiquitous heat equation, corresponding to the nondegenerate case when $\sA = \Delta$ ($Q = I_N$, $B = O_N$). When $\sA = \Delta - <X,\nabla>$ ($Q=I_N$, $B = - I_N$) one has the Ornstein-Uhlenbeck operator, of great interest in probability, see e.g. \cite{DZ}. Our primary motivating example, however, is the degenerate Kolmogorov operator, which arose in the seminal paper \cite{Kol} on Brownian motion and the theory of gases. Denote by $(X,t)=(v,x,t)$ the generic point in $\R^{N+1}$ with $N=2n$. With the choices
$Q = \begin{pmatrix} I_n & 0_n\\ 0_n& 0_n\end{pmatrix}$, and $B = \begin{pmatrix} 0_n & 0_n\\ I_n & 0_n\end{pmatrix}$, the operator $\K$ in \eqref{K0} becomes
\begin{equation}\label{kolmo0}
\K u = \Delta_v u+ <v,\nabla_x u > - \de_t u.
\end{equation}
Clearly, \eqref{kolmo0} fails to be parabolic since it is missing the diffusive term $\Delta_x u$, but it is easily seen to satisfy H\"ormander's finite rank condition, and thus $\K$ is hypoelliptic. We note that, remarkably, Kolmogorov had already proved this fact thirty years prior to \cite{Ho} by exhibiting the following explicit fundamental solution for \eqref{kolmo0} 
\begin{equation}\label{kolmofs0}
p(X,Y,t) = \frac{c_n}{t^{2n}} \exp\left\{- \frac 1t \left(|v-w|^2 + \frac 3t <v-w,y-x-tv> + \frac{3}{t^2} |x- y +tv|^2\right)\right\}.
\end{equation}
Since \eqref{kolmofs0} is $C^\infty$ off the diagonal, it follows that \eqref{kolmo0} is hypoelliptic.

The class of partial differential operators \eqref{K0} has been intensively studied over the past thirty years, and thanks to the work of many people a lot is known about it. Nonetheless, some fundamental aspects presently remain elusive, such as Sobolev or isoperimetric inequalities, a Calder\'on-Zygmund theory (but for some interesting progress in this direction, see \cite{BCLP}), and one of local and nonlocal minimal surfaces. The difficulties with these hypoelliptic operators stem from the fact that the drift term in \eqref{K0} mixes the variables inextricably and this complicates the geometry considerably. This is already evident at the level of the model equation \eqref{kolmo0} and its probability transition kernel \eqref{kolmofs0}. Unlike what happens for H\"ormander operators of the form $\sum_{i=1}^m Y_i^2 - \p_t$ (see, e.g., \cite{VSC}, \cite{Gjems} and the references therein), where there is only one intrinsic distance $d(x,y)$ that controls the geometry for all times, for \eqref{kolmofs0} there is a one-parameter family of non-symmetric pseudo-distances $d_t(X,Y)$ that drive the evolution. 
Such intertwined geometries are reflected in the large time behaviour of H\"ormander's fundamental solution of \eqref{K0}. In many respects such behaviour parallels the diverse situations that one encounters in the Riemannian setting when passing from positive to negative curvature. In general, the relevant volume function is not power-like in $t$ and need not be doubling. A detailed description of the different behaviours is contained in \cite{GT}.

Having said this, we turn to the focus of the present note. Our primary objective is to establish a sufficiently robust nonlocal calculus for a subclass of the hypoelliptic operators \eqref{K0} that includes \eqref{kolmo0} as a special case. In the forthcoming works \cite{GT}, \cite{GTiso}, starting from such calculus, we will establish some new Sobolev and isoperimetric inequalities. To be specific, our focus is on operators which, besides H\"ormander's hypoellipticity condition $K(t)>0$ for all $t>0$, also satisfy the assumption on the drift
\begin{equation}\label{trace}
\operatorname{tr} B \ge 0.
\end{equation} 
Let us notice explicitly that such hypothesis includes, as special cases, the heat equation or \eqref{kolmo0}, for both of which we have $\operatorname{tr} B = 0$. But it leaves out examples such as the Ornstein-Uhlenbeck operator mentioned above, or the equation 
\begin{equation*}
\K u = \de_{xx} u  - 2(x+y) \de_x u + x \de_y u - \de_t u = 0,
\end{equation*}
which arises in the Smoluchowski-Kramers' approximation of Brownian motion with friction, see \cite{Bri} and \cite{Fre}. For the former we have $B = - I_N$, while for the latter one has $B = \begin{pmatrix} -2 & -2\\ 1 & 0\end{pmatrix}$. 

To understand the role of \eqref{trace} in the present work we recall that the Cauchy problem $\K u = 0$ in $\R^{N+1}_+$, $u(X,0) = f$, admits a unique solution for $f\in \So$, see Theorem \ref{T:hor}. This generates a strongly continuous semigroup $\{P_t\}_{t>0}$ on $\Lp$ defined by 
\[
P_t f(X) = \int_{\RN} p(X,Y,t) f(Y) dt,
\]
where $p(X,Y,t)$ is the transition distribution constructed by H\"ormander in \cite{Ho}, see also \eqref{PtKt}. However, the spectral properties of this semigroup dramatically change depending on the sign of $\tr\ B$. The assumption \eqref{trace} guarantees that $\{P_t\}_{t>0}$ is contractive on $\Lp$, and this aspect plays a pervasive role in the present paper. We will return to the analysis of \eqref{K0} in the case $\operatorname{tr} B<0$ in a forthcoming study. In this work, regardless of the sign of $\operatorname{tr} B$, we solve the extension problem in $L^\infty$ (see Theorem \ref{T:epKinfty} and Theorem \ref{T:extLinfty}). For the $L^p$-case, instead, we shall assume \eqref{trace}.

\vskip 0.4cm

To put our results in the proper perspective we mention that the study of nonlocal equations is very classical, stretching back to the seminal works of M. Riesz \cite{R1, R2} on the fractional powers of the Laplacian $(-\Delta)^s$ and the wave operator $(\de_{tt} - \Delta)^s$. A semigroup based fractional calculus for closed linear operators was first introduced by Bochner in his visionary note \cite{Bo}, see also Feller's work \cite{Fe}. Phillips showed in \cite{Phi} that one could embed these approaches into a more
general one based on the Kolmogorov-Levy representation theorem for infinitely divisible distributions. In \cite{B} Balakrishnan introduced a new fractional calculus that extended 
the previous contributions to situations in which the relevant operator does not necessarily generate a semigroup. For a given closed operator $A$ on a Banach space $X$, under the assumption that $||\la R(\la,A)||\le M$ for $\la>0$ (there exist operators $A$ which satisfy such hypothesis but do not generate a semigroup), he constructed the fractional powers of $A$ by the beautiful formula 
\begin{equation}\label{bala0}
A^\alpha x  = - \frac{\sin(\pi \alpha)}{\pi} \int_{0}^{\infty} \lambda^{\alpha-1} R(\lambda,A) A x d\la,\ \ \ \ \ \ \ \ 0<\Re \alpha < 1,
\end{equation}
see \cite[(2.1)]{B}. When $A$ does generate a strongly continuous semigroup $\{T(t)\}_{t>0}$ on $X$, then it is well-known that \eqref{bala0} can also be expressed as follows
\begin{equation}\label{bala02}
A^\alpha x  = - \frac{\alpha}{\G(1-\alpha)} \int_{0}^{\infty} t^{-\alpha-1} [T(t) x - x] dt,\ \ \ \ \ \ \ \ 0<\Re \alpha < 1.
\end{equation}
Similarly to the existing literature in the classical setting $\K = \Delta - \p_t$, see \cite[(5.84) on p. 120]{SKM}, Balakrishnan's formula \eqref{bala02} is the starting point of our analysis. The gist of our work is to develop those mathematical tools that allow to successfully push the ideas in \cite{CS}
to the class of degenerate hypoelliptic equations  \eqref{K0}.
 
With $\sA$ as in \eqref{K0}, we use \eqref{bala02} and the semigroup $\{P_t\}_{t>0}$ to define the fractional powers on functions $f\in \So(\RN)$ by the pointwise formula 
\begin{equation}\label{As0}
\As f(X) = - \frac{s}{\G(1-s)} \int_{0}^{\infty} t^{-s-1} [P_t f(X) - f(X)] dt,\ \ \ \ \ \ \ \ 0<s < 1.
\end{equation}  
Since we also want to have a nonlocal calculus for the time-dependent operator $\K$, we introduce on a function $u\in \mathscr S(\RNu)$ what we call the \emph{H\"ormander evolutive semigroup} 
$$P^\K_\tau u(X,t) \overset{def}{=} \int_{\R^N} p(X,Y,\tau) u(Y,t-\tau) dY,\qquad (X,t)\in\RNu,\,\,\tau>0.$$
The notion of \emph{evolution semigroup} is well-known in dynamical systems, and the reader should see \cite{CL} in this respect. 
Using $\{P^\K_\tau\}_{\tau>0}$ we define on a function $u\in \mathscr S(\RNu)$,
\begin{equation}\label{As02}
\left(-\K\right)^s u(X,t) = - \frac{s}{\G(1-s)} \int_{0}^{\infty} \tau^{-s-1} [P^\K_\tau u(X,t) - u(X,t)] d\tau,\ \ \ \ \ \ \ \ 0<s < 1.
\end{equation}  

Having in mind the development of the program mentioned above, with definitions \eqref{As0} and \eqref{As02} in hand we turn the attention to the basic question of characterizing these nonlocal operators as traces of suitable Bessel processes. In probability this was first introduced by Molchanov and Ostrovskii in \cite{MO} for symmetric stable processes. 
But it was not until the celebrated 2007 extension paper of Caffarelli and Silvestre \cite{CS} that such idea became a powerful tool in analysis and geometry. Their work has allowed to convert problems involving the nonlocal operator $(-\Delta)^s$ in $\Rn$, into problems in $\R^{n}\times(0,\infty)$ involving the (local) partial differential equation of degenerate type
$$\begin{cases}
\operatorname{div}_{(x,z)}\left(z^a \nabla_{(x,z)} U\right) = 0,
\\
U(x,0) = u(x).
\end{cases}$$
One remarkable aspect of this procedure is represented by the limiting relation
$$- \frac{2^{-a} \Gamma\left(\frac{1-a}2\right)}{\Gamma\left(\frac{1+a}2\right)}  \underset{z\to 0^+}{\lim} \paa U(x,z) = (-\Delta)^s u(x),$$
where the parameters $0<s<1$ and $a\in (-1,1)$ are connected by the equation $a = 1-2s$ (hereafter, for $\ell>0$ we indicate with $\G(\ell) = \int_0^\infty \tau^\ell e^{-\tau} \frac{d\tau}\tau$ Euler's gamma function evaluated at $\ell$). 
 
In the present paper we establish results analogous (at least on the formal level) to Caffarelli and Silvestre's (see also \cite{ST10}) for the nonlocal operators $(-\mathscr K)^s$ and $\As$. Precisely, we first solve the extension problem for \eqref{As02}, and then we combine it with Bochner's subordination to obtain a corresponding solution for \eqref{As0}. The construction of the relevant Poisson kernels is based on fairly explicit formulas which involve H\"ormander's fundamental solution \eqref{PtKt}, and provide a flexible and robust tool for the theory developed in \cite{GT}, \cite{GTiso}. 

As a final comment we mention that the novelty of our work is in the treatment of the genuinely degenerate hypoelliptic operators in \eqref{K0} when $Q\ge 0$ and $B\not = O_N$. In fact, in the nondegenerate case when $\sA = \Delta$, and thus no drift is present, in a remarkable 1968 paper Frank Jones first solved the extension problem for the fractional heat equation $(\p_t -\Delta)^{1/2}$ and constructed an explicit Poisson kernel for the extension operator, see \cite[(2.1) in Sec.2]{Jo} and the subsequent formulas. Such Poisson kernel was recently generalised by Nystr\"om and Sande \cite{NS} and Stinga and Torrea \cite{ST} to the case of fractional powers with $s\not= 1/2$. Our results can be seen as a far-reaching extension of these results to the much larger class \eqref{K0}, under the hypothesis \eqref{trace}. In connection with extension problems for sub-Laplacians in Carnot groups and for sum of squares of H\"ormander vector fields we mention \cite{FF, G18b}, and \cite{FGMT} for a related but more geometric result in the CR setting. 

\vskip 0.4cm

The organization of the paper is as follows. In Section \ref{S:ks} we collect some well-known properties of the H\"ormander semigroup which are used throughout the rest of the paper. We also introduce the evolutive H\"ormander semigroup $\{P^\K_\tau\}_{\tau>0}$ and extend to the latter the results for $\{P_t\}_{t>0}$. This allows us to define in Section \ref{S:fracK} the fractional powers $(-\K)^s$ and the related extension problem, see Definition \ref{D:epKa}. In Proposition \ref{P:Ga} we introduce the Neumann fundamental solution, and in Definition \ref{D:PaK} the Poisson kernel for the extension problem. Section \ref{S:epKs} is devoted to the proof of Theorems \ref{T:epKinfty} and \ref{T:epKL2}. With these two results we prove the validity of the Dirichlet-to-Neumann condition respectively in $L^\infty$ and, under the additional assumption \eqref{trace}, in $L^p$. In Section \ref{S:fracA} we study the nonlocal operator $(- \mathscr A)^s$, where $\mathscr A$ is the diffusive part in \eqref{K0}. The main result of this section is Theorem \ref{T:extLinfty}, where we solve the relevant extension problem. 

\subsection{Notation} 

All the function spaces in this paper are based either on $\RN$ or on $\RNu$. For spaces on $\RN$ we will routinely avoid reference to the ambient space, for those on $\RNu$ we will explicitly mention the ambient space. For instance, the Schwartz space of rapidly decreasing functions in $\RN$ will be denoted by $\So$, whereas $\S$ denotes the Schwartz space in $\RNu$. The same convention applies to the $\Lp$-spaces for $1\le p \le \infty$. The norm in $\Lp$ will be denoted by $||\cdot||_p$ whenever there is no confusion with the ambient space. We will indicate with $\Lii$ the Banach space of the $f\in C(\RN)$ such that $\underset{|X|\to \infty}{\lim}\ |f(X)| = 0$ with the norm $||\cdot||_\infty$. If $T:\Lp\to L^q$ is a bounded linear map, we will indicate with $||T||_{p\to q}$ its operator norm. If $ q =p$, the spectrum of $T$ on $\Lp$ will be denoted by $\sigma_p(T)$, the resolvent set by $\rho_p(T)$, the resolvent operator by $R(\la,T) = (\la I - T)^{-1}$. We adopt the convention that $a/\infty = 0$ for any $a\in \R$.


\section{The H\"ormander semigroup $\{P_t\}_{t>0}$}\label{S:ks}

In this section we collect some well-known properties of the semigroup associated with \eqref{K0} which will be used throughout the rest of the paper. The reader should see the  works \cite{Ku}, \cite{K82}, \cite{GL}, \cite{LP}, \cite{L}, \cite{Pritorino}, \cite{Me}, \cite{LB}, \cite{Pristudia} and \cite{AB}. 
As we have mentioned in the introduction, the starting point is the following result from \cite{Ho}.

\begin{theorem}[H\"ormander]\label{T:hor}
Given $Q$ and $B$ as in \eqref{K0}, for every $t> 0$ consider the \emph{covariance matrix} \eqref{Kt}.
Then, the operator $\K$ is hypoelliptic if and only if $\det K(t)>0$ for every $t>0$.
In such case, given $f \in \So$, the unique solution to the Cauchy problem $\K u = 0$ in $\R^{N+1}_+$, $u(X,0) = f$, is given by
$u(X,t) = \int_{\R^N} p(X,Y,t) f(Y) dY,$
where
\begin{equation}\label{PtKt}
p(X,Y,t) = (4\pi)^{- \frac N2} \left(\operatorname{det}(t K(t))\right)^{-1/2} \exp\left( - \frac{<K(t)^{-1}(Y-e^{tB} X ),Y-e^{tB} X >}{4t}\right).
\end{equation}
\end{theorem}

Throughout the paper we always assume that $K(t)>0$ for every $t>0$. One should keep in mind that the hypoellipticity of \eqref{K0} can be expressed in a number of different ways, see \cite{LP}. It was noted in the same paper that the operator $\K$ is invariant with respect to the following non-Abelian group law $(X,s)\circ (Y,t) = (Y+ e^{-tB}X,s+t)$. Endowed with the latter the space $\R^{N+1}$ becomes a non-Abelian Lie group. 
In what follows it will be convenient to also have the following alternative expression for the kernel $p(X,Y,t)$ in \eqref{PtKt} (see, e.g., \cite{Ku, LP}):
\begin{equation}\label{pcomeG}
p(X,Y,t)= (4\pi)^{- \frac N2} e^{- t \operatorname{tr}(B)} \left(\operatorname{det}(C(t))\right)^{-1/2} \exp\big(- \frac{<C(t)^{-1} X-e^{-tB} Y,X-e^{-tB} Y>}{4}\big),
\end{equation}
where $C(t) = \int_0^t e^{-sB} Q e^{-sB^\star} ds$.
Notice that $C(t)^\star = C(t)$ and since
\begin{equation}\label{KC}
t K(t) = e^{tB} C(t) e^{tB^\star},
\end{equation}
it is clear that $K(t)>0$ if and only if $C(t)>0$. Now, given a function $f\in \So$ we define 
\begin{align}\label{pt}
P_t f(X) \overset{def}{=} \int_{\R^N} p(X,Y,t) f(Y) dY.
\end{align}
In the next two lemmas we collect the main properties of $\{P_t\}_{t>0}$. These results are well-known to the experts, but we include them for completeness.

\begin{lemma}\label{L:invS}
For any $t>0$ we have: 
\begin{itemize}
\item[(a)] $\sA(\So)\subset \So$ and $P_t(\So) \subset \So$;
\item[(b)] For any $f\in \So$ and $X\in \RN$ one has $\frac{\de}{\de t} P_t f(X) = \mathscr A P_t f(X)$; 
\item[(c)] For every $f\in \So$ and $X\in \RN$ the commutation property is true
\begin{equation}\label{eqPtA}
\mathscr A P_t f(X) = P_t \mathscr A  f(X).
\end{equation}
\end{itemize}
\end{lemma}

\begin{proof}
 (a) The first part is obvious. For the second part it suffices to show that $\widehat{P_t f}\in \So$, and this follows from the following formula 
\begin{equation}\label{FTPt}
\widehat{P_t f}(\xi) = e^{-t \operatorname{tr} B}  e^{- 4 \pi^2 <C(t)\xi,\xi>} \hat{f}(e^{-tB^\star} \xi).
\end{equation}
(b) Easily follows from differentiating \eqref{FTPt} with respect to $t$, and using the following formula,
$$
\widehat{\sA f}(\xi) = - \left[<B^\star \xi, \nabla_\xi \hat f(\xi)> + \left(4 \pi^2 <Q\xi,\xi> + \operatorname{tr} B\right) \hat f(\xi)\right]$$
in combination with \eqref{FTPt}. (c) By (a), \eqref{eqPtA} is equivalent to showing that $\widehat{\sA P_t f} = \widehat{P_t \sA f}$ for $f\in \So$. 
After a routine computation, this is shown equivalent to the identity between the two symmetric quadratic forms
\[
 <e^{-tB} Q e^{-tB^\star} \xi,\xi>\ =\  <Q\xi,\xi> -  <B C(t) \xi,\xi> - <C(t) B^\star \xi,\xi>,\ \ \ \ \ \ \  \xi\in \RN,\ t>0.
\]
This is true as a consequence of the matrix identity
$$e^{-tB}Q e^{-tB^*}= Q  - B C(t) - C(t) B^*,\ \ \ \ \ \ \ \ t>0,$$
that can be verified by noting that both sides vanish at $t = 0$ and they have the same derivative in $t$ (see also \cite[equation (4.6)]{AT}).

\end{proof}

We observe the following simple fact. 

\begin{lemma}\label{L:Linfty}
One has: (1) $P_t : L^\infty_0 \to L^\infty_0$ for every $t>0$; (2) $\So$ is dense in $\Lii$.
\end{lemma}

We next collect some known results concerning the action of $\{P_t\}_{t>0}$ on the spaces $\Lp$, see \cite{Me} and \cite{LB}.

\begin{lemma}\label{L:Pt}
The following properties hold:
\begin{itemize}
\item[(i)] For every $X\in \RN$ and $t>0$ we have
$P_t 1(X) = \int_{\RN} p(X,Y,t) dY = 1$;
\item[(ii)] $P_t:L^\infty \to L^\infty$ with $||P_t||_{L^\infty\to L^\infty} \le 1$;
\item[(iii)] For every $Y\in \RN$ and $t>0$ one has
$
\int_{\RN} p(X,Y,t) dX = e^{- t \operatorname{tr} B}.
$
\item[(iv)] Let $1\le p<\infty$, then $P_t:L^p \to L^p$ with $||P_t||_{L^p\to L^p} \le e^{-\frac{t \operatorname{tr} B}p}$. If $\operatorname{tr} B\ge 0$, $P_t$ is a contraction on $L^p$ for every $t>0$;
\item[(v)] [Chapman-Kolmogorov equation]
for every $X, Y\in \R^N$ and $t>0$ one has
$$
p(X,Y,s+t)  = \int_{\R^N} p(X,Z,s) p(Z,Y,t) dZ.
$$
Equivalently, one has $P_{t+s} = P_t \circ P_s$ for every $s, t>0$.
\end{itemize}
\end{lemma}

We note that it was shown in \cite{L} that $\{P_t\}_{t>0}$ is not a strongly continuous semigroup in the space of uniformly continuous bounded functions in $\RN$, but this fact will have no bearing on our results since we are primarily concerned with the action of the H\"ormander semigroup on $L^p$, when $1\le p<\infty$, and on the replacement space $\Lii$ when $p = \infty$. In this respect, we begin with a simple but quite useful lemma. 
 
\begin{lemma}\label{L:Lprate}
Let $1\le p \le \infty$. Given any $f\in \So$ for any $t\in [0,1]$ we have
\[
||P_t f - f||_{p} \le ||\mathscr A f||_{p}\ \omega(t),
\]
where $\omega(t)\le \max\{1,e^{-\frac{\operatorname{tr} B}p}\}\ t$.
\end{lemma}

\begin{proof}
By Lemma \ref{L:invS}, part (b) and the commutation identity \eqref{eqPtA}, we have 
for any $f\in \So$, 
\begin{align*}
P_t f(X) - f(X) & = \int_0^t \frac{d}{d\tau} P_\tau f(X) d\tau = \int_0^t \mathscr A P_\tau f(X) d\tau = \int_0^t P_\tau  \mathscr A f(X) d\tau.
\end{align*}
This gives for any $0\le t \le 1$,
\[
||P_t f - f||_{p} \le \int_0^t ||P_\tau  \mathscr A f||_p d\tau \le ||\mathscr A f||_{p} \int_0^t  e^{- \tau \frac{\operatorname{tr} B}p} d\tau = ||\mathscr A f||_{p}\  \omega(t),
\]
where in the second inequality we have used (ii) and (iv) of Lemma \ref{L:Pt}.

\end{proof}

\begin{corollary}\label{C:Ptpzero}
Let $1\le p< \infty$. For every $f\in L^p$, we have
$||P_tf-f||_{p}\rightarrow 0$ as $t \to 0^+.$ Consequently, $\{P_t\}_{t>0}$ is a strongly continuous semigroup on $\Lp$. The same is true when $p = \infty$, if we replace $L^\infty$ by the space $\Lii$.  
\end{corollary}

\begin{proof}
The first part of the statement follows immediately from the density of $\So$ in $L^p$ and from Lemmas \ref{L:Lprate} and \ref{L:Linfty}). 
The second part is a standard consequence of the former, see e.g. \cite[Proposition 1.3]{EN}.

\end{proof}

\begin{remark}\label{R:infty}
The reader should keep in mind that from this point on when we consider $\{P_t\}_{t>0}$ as a strongly continuous semigroup in $\Lp$, when $p = \infty$ we always mean that $\Lii$ must be used instead of $\Li$.
\end{remark}

Denote by $(\sA_p,D_p)$ the infinitesimal generator of the semigroup $\{P_t\}_{t>0}$ on $L^p$ with domain  
\[
D_p = \left\{f\in L^p\mid \sA_p f \overset{def}{=} \underset{t\to 0^+}{\lim}\ \frac{P_t f - f}{t}\ \text{exists in }\ L^p\right\}.
\]
One knows that $(\sA_p,D_p)$ is closed and densely defined (see \cite[Theorem 1.4]{EN}). 

\begin{corollary}\label{C:lp}
We have $\So\subset D_p$. Furthermore, $\sA_p f = \sA f$ for any $f\in \So$, and $\So$ is a core for $(\sA_p,D_p)$.
\end{corollary}

\begin{proof}

For any $f\in \So$ we obtain from \eqref{eqPtA}: $\frac{P_t f - f}{t} - \sA f = \frac 1t \int_0^t \left[P_s \sA f - \sA f\right] ds$.
An application of Minkowski's integral inequality and Lemma \ref{L:Lprate} (keeping in mind that $\sA f\in \So$ as well) give
\[
\left\|\frac{P_t f - f}{t} - \sA f\right\|_{p} \le \frac 1t \int_0^t ||P_s \sA f - \sA f||_{p} ds \le C ||\sA^2 f||_{p} \ t.
\]
This shows  that $\So \subset D_p$, and moreover the two linear operators $\sA_p$  and $\sA$ coincide on the dense subspace $\So$. Finally, the fact that $\So$ is a core for $(\sA_p,D_p)$ follows from the second part of (a) in Lemma \ref{L:invS} and the fact that $\So$ is dense in $\Lp$, see \cite[Proposition 1.7]{EN}.

\end{proof}

\begin{remark}\label{R:id}
From now on for a given $p\in [1,\infty]$ with a slight abuse of notation we write $\sA : D_p\to \Lp$ instead of $\sA_p$. In so doing, we must keep in mind that $\sA$ actually indicates the closed operator $\sA_p$ that, thanks to Corollary \ref{C:lp}, coincides with the differential operator $\sA$ on $\So$. Using this identification we will henceforth say that $(\sA,D_p)$ is the infinitesimal generator of the semigroup $\{P_t\}_{t>0}$ on $\Lp$.
\end{remark}
Up to now we have not made use of the assumption \eqref{trace}. In the next lemma we change course.

\begin{lemma}\label{L:specter}
Assume that \eqref{trace} be in force, and let $1\le p \le \infty$. Then: 
\begin{itemize}
\item[(1)] For any $\la\in \mathbb C$ such that $\Re \la >0$, we have $\la\in \rpp$;
\item[(2)] If $\la\in \mathbb C$ such that $\Re \la >0$, then $R(\la,\sA)$ exists and for any $f\in \Lp$ it is given by the formula $R(\la,\sA) f = \int_0^\infty e^{-\la t} P_t f\ dt$;
\item[(3)] For any $\Re \la > 0$ we have
$||R(\la,\sA)||_{p\to p} \le \frac{1}{\Re \la}$.
\end{itemize}
\end{lemma}

We omit the proof of Lemma \ref{L:specter} since it is a direct consequence of  (ii), (iv) in Lemma \ref{L:Pt}, and of \cite[Theorem 1.10]{EN}.




In semigroup theory a procedure for forming a new semigroup from a given one is that of evolution semigroup, see \cite{CL}. In what follows we exploit  this idea to introduce a new semigroup that will be used as a building block for: (1) defining the fractional powers of the operator $\K$ in \eqref{K0} above; (2) solve the extension problem for such nonlocal operators.
Henceforth, we use the notation $\RNu$ to indicate the space $\RN \times \R$ with respect the variables $(X,t)$. 

\begin{definition}\label{D:eks}
With $p(X,Y,\tau)$ as in \eqref{PtKt}, we define the \emph{evolutive H\"ormander semigroup} on a function $u\in \mathscr S(\RNu)$ as
\begin{equation}\label{eks}
P^\K_\tau u(X,t) \overset{def}{=} \int_{\R^N} p(X,Y,\tau) u(Y,t-\tau) dY,\qquad (X,t)\in\RNu,\,\,\tau>0.
\end{equation}
\end{definition}
We observe that if we let $\Lambda_h u(X,t) = u(X,t+h)$, then \eqref{eks} can be also written as $P^\K_\tau u = P_\tau(\Lambda_{-\tau} u)$.

\begin{lemma}\label{L:eks}
If for $u\in \S$ we define $v(X,t;\tau) = P^\K_\tau u(X,t)$,
then $v\in C^\infty(\RNu\times (0,\infty))$
and it solves the Cauchy problem
$$\begin{cases}
\de_\tau v = \K v \ \ \ \ \ \ \ \ \ \ \ \ \ \ \ \  \  \text{in}\ \RNu\times (0,\infty),
\\
v(X,t;0) = u(X,t)\ \ \ \ \ \ \ (X,t)\in \RNu.
\end{cases}$$
\end{lemma}
\begin{proof}
First of all, from the properties of $P_\tau$, it is easy to verify that $v(X,t;\tau)$ tends to $u(X,t)$ as $\tau\rightarrow 0^+$. Moreover, the assumption that $u\in \S$ implies that it has bounded time-derivatives of any order. This fact, together with the Gaussian behavior of the kernel $p(X,Y,\tau)$ (and of its derivatives), allows to differentiate under the integral sign for $\tau>0$: for more details, the reader can find in \eqref{Xderp} an explicit computation of the first derivatives of $p$. In particular, $v$ is $C^\infty(\RNu\times (0,\infty))$. Finally $\de_\tau v = \K v$ since, for positive $\tau$, we have
$$\de_\tau v + \de_t v =\int_{\R^N} \de_\tau p(X,Y,\tau) u(Y,t-\tau) dY = \int_{\R^N} \mathscr A p(X,Y,\tau) u(Y,t-\tau) dY=\mathscr A v.$$
\end{proof}

We will need the counterpart of Lemmas \ref{L:invS} and \ref{L:Pt} for the semigroup $\{P^\K_\tau\}_{\tau>0}$.

\begin{lemma}\label{L:invSK}
For any $t>0$ we have: 
\begin{itemize}
\item[(a)] $\K(\S)\subset \S$ and $P^\K_\tau(\S) \subset \S$;
\item[(b)] For any $u\in \S$ and $(X,t)\in \RNu$ one has $\frac{\p}{\p_\tau} P^\K_\tau u(X,t) = \K P^\K_\tau u(X,t)$; 
\item[(c)] For every $u\in \S$ and $(X,t)\in \RNu$ the commutation property is true
$$
\K P^\K_\tau u(X,t) = P^\K_\tau \K  u(X,t).
$$
\end{itemize}
\end{lemma}
\begin{proof}
 (a) The first part is obvious. For the second part it suffices to show that $\widehat{P^\K_\tau \psi}\in \S$ if $\psi\in\S$, and this follows from the following formula
 \[
\widehat{P^\K_\tau \psi}(\xi,\sigma) = e^{-\tau \operatorname{tr} B}  e^{- 4 \pi^2 <C(\tau)\xi,\xi>} e^{-2\pi i \tau \sigma} \hat{\psi}(e^{-\tau B^\star} \xi, \sigma).
\]
(b) Is a consequence of Lemma \ref{L:eks}. (c) Follows from the commutation property $\sA P_t=P_t\sA$ proved in Lemma \ref{L:invS}, and from the relations $P^\K_\tau u = P_\tau(\Lambda_{-\tau} u)$, $\K \Lambda_{-\tau}= \Lambda_{-\tau} \K$.

\end{proof}

\begin{lemma}\label{L:PtK}
The following properties hold:
\begin{itemize}
\item[(i)] For every $(X,t)\in \RNu$  and $\tau>0$ we have $P^\K_\tau 1(X,t)=1$;
\item[(ii)] We have $P^\K_{\tau+s} = P^\K_\tau \circ P^\K_s$ for every $s, \tau>0$.
\item[(iii)] $P^\K_\tau:L^\infty(\RNu) \to L^\infty(\RNu)$ with $||P^\K_\tau||_{L^\infty\to L^\infty} \le 1$;
 \item[(iv)] Let $1\le p<\infty$, then $P^\K_\tau:L^p(\RNu) \to L^p(\RNu)$ with $||P^\K_\tau||_{L^p\to L^p} \le e^{-\frac{\tau \operatorname{tr} B}p}$. 
 \item[(v)] If \eqref{trace} holds, $\{P^\K_\tau\}_{\tau>0}$ is a strongly continuous semigroup of contractions on $\Lp(\RNu)$.
\end{itemize}
\end{lemma}
\begin{proof} The proof of the desired statements easily follows from Definition \ref{D:eks}, the identity $P^\K_\tau u = P_\tau(\Lambda_{-\tau} u)$ and Lemma \ref{L:Pt}. We only provide the details of (iii). Using the above mentioned ingredients and Tonelli's theorem we have for any $u\in\S$
\begin{align*}
& ||P^\K_\tau u||_{L^p(\RNu)} = \left(\int_\R ||P_\tau(\Lambda_{-\tau} u(\cdot,t))||^p_{L^p(\RN)} dt\right)^{1/p} \le e^{-\tau \frac{\operatorname{tr} B}p} \left(\int_\R ||\Lambda_{-\tau} u(\cdot,t)||^p_{L^p(\RN)} dt\right)^{1/p} \\
&=  e^{-\tau \frac{\operatorname{tr} B}p} \left(\int_\R ||u(\cdot,t)||^p_{L^p(\RN)} dt\right)^{1/p} =  e^{-\tau \frac{\operatorname{tr} B}p} ||u||_{L^p(\RNu)}.
\end{align*}
\end{proof}

We conclude the section with the analogue of Lemma \ref{L:Lprate} for the semigroup $\{P^\K_\tau\}_{\tau>0}$. Its proof proceeds along the same lines exploiting Lemma \ref{L:invSK} and Lemma \ref{L:PtK}.

\begin{lemma}\label{P:Kdiff}
Let $1\le p \le \infty$. Given any $f\in \S$ for any $\tau\in [0,1]$ we have
\[
||P^\K_t f - f||_{p} \le ||\K f||_{p}\ \omega(\tau),
\]
where $\omega(\tau)\le \max\left\{1,e^{-\frac{\operatorname{tr} B}p}\right\}\ \tau$.
\end{lemma}



\section{The nonlocal operators $\As$, $(-\K)^s$ and their extension problems}\label{S:fracK}

Fix $0<s<1$. With the results of the previous section in hand we are now ready to introduce the definition of the nonlocal operators $\As$ and $(-\K)^s$.

\begin{definition}\label{D:Ks}
For any $\vf\in \So$ we define the nonlocal operator by the following pointwise formula
\begin{align}\label{As}
(-\mathscr A)^s \vf(X) & =  - \frac{s}{\G(1-s)} \int_0^\infty t^{-s-1} \left[P_t \vf(X) - \vf(X)\right] dt,\qquad X\in\RN.
\end{align}
Similarly, for $u \in\S$ and $(X,t)\in \RNu$, we define 
\begin{equation}\label{Ks}
\left(-\K\right)^s u(X,t) = - \frac{s}{\G(1-s)} \int_0^\infty \tau^{-1-s} \left[P^\K_\tau u(X,t)- u(X,t)\right] d\tau.
\end{equation}
\end{definition} 

\begin{remark}\label{R:KsAs}
We note explicitly that when $u(X,t) = u(X)$, then we obtain from \eqref{eks}
\[
P^\K_\tau u(X,t) \overset{def}{=} \int_{\R^N} p(X,Y,\tau) u(Y) dY = P_\tau u(X),\qquad (X,t)\in\RNu,\,\,\tau>0.
\]
In such case, formulas \eqref{Ks} and \eqref{As} give
$$
\left(-\K\right)^s u(X,t) = - \frac{s}{\G(1-s)} \int_0^\infty \tau^{-1-s} \left[P_\tau u(X)- u(X)\right] d\tau = (-\mathscr A)^s u(X).
$$
\end{remark}

As a first observation we note that the integrals in the right-hand side of \eqref{As}, \eqref{Ks} are convergent. To check this, for instance, for \eqref{Ks}, write
\begin{align*}
& \int_0^\infty \tau^{-1-s} \left[P^\K_\tau u(X,t)- u(X,t)\right] d\tau = 
\int_0^1 \tau^{-1-s} \left[P^\K_\tau u(X,t)- u(X,t)\right] d\tau
\\
& + \int_1^\infty \tau^{-1-s} \left[P^\K_\tau u(X,t)- u(X,t)\right] d\tau.
\end{align*}
In the second integral we use (ii) in Lemma \ref{L:PtK} which gives
\begin{align*}
& \tau^{-1-s} \left|P^\K_\tau u(X,t)- u(X,t)\right| \le \tau^{-1-s} \left(||P^\K_\tau u||_{L^\infty(\RNu)} + ||u||_{L^\infty(\RNu)}\right)
\\
& \le 2 ||u||_{L^\infty(\RNu)} \tau^{-1-s}\in L^1(1,\infty).
\end{align*}
For the first integral we use the crucial Lemma \ref{P:Kdiff}, that implies
\begin{align*}
& \tau^{-1-s} \left|P^\K_\tau u(X,t)- u(X,t)\right| \le \tau^{-1-s}  ||P^\K_\tau u - u||_{L^\infty(\RNu)} \le C \tau^{-s} \in L^1(0,1).
\end{align*}

\begin{remark}\label{R:L2K}
We emphasise that, because of the large-time behaviour of the semigroups $P_t$ and $P^\K_\tau$, when $1\leq p<\infty$ it may not be true in general that the function defined by the right-hand side of \eqref{As}, \eqref{Ks} be in $L^p$! We note however that, when \eqref{trace} holds, we can appeal to (iv) in Lemma \ref{L:Pt}, or (v) of Lemma \ref{L:PtK}, to show, by arguments similar to those above, that the equations \eqref{As}, \eqref{Ks} do define $\Lp$ functions.
\end{remark}

With Definition \ref{D:Ks} in hands we next introduce the extension problem for the nonlocal operator $(-\K)^s$. Following \cite{CS}, this is going to be a Dirichlet problem in one dimension up. Precisely, on the half-line $\R^+ = (0,\infty)$ with variable $z$ we consider the Bessel operator
$\Ba = \frac{\p^2}{\p z^2} + \frac az \frac{\p}{\p z}$ with $a> -1$.
We define the \emph{extension operator} as the following second-order partial differential operator in $\RNu\times (0,\infty)$
\begin{equation}\label{Ka}
\K_a = z^a(\K + \Ba) = z^a(\mathscr A + \Ba - \p_t).
\end{equation}

\begin{definition}\label{D:epKa}
The \emph{extension problem} consists in finding, for a given $u\in \S$, a function $U\in C^\infty(\RNu \times (0,\infty))$ such that
\begin{equation}\label{epKa}
\begin{cases}
\K_a U = 0\ \  \text{in}\ \RNu\times (0,\infty),
\\
U(X,t,0) = u(X,t).
\end{cases}
\end{equation}
\end{definition} 

In order to solve the problem \eqref{epKa} we are going to construct an appropriate Poisson kernel for it. Since the Bessel process  plays a pivotal role in what follows, we recall some well-known properties of the latter.
On the half-line $(0,\infty)$ we consider the Cauchy problem for $\Ba$ with the Neumann boundary condition (this corresponds to reflected Brownian motion, as opposed to killed Brownian motion, when a Dirichlet condition is imposed):  
\[
\begin{cases}
\p_t u  - \Ba u = 0,\ \ \ \ \ \ \ \text{in} \ (0,\infty)\times (0,\infty),
\\
u(z,0) = \vf(z),\ \ \ \ \ \ \ \ z\in (0,\infty),
\\
\underset{z\to 0^+}{\lim} \paa u(z,t) = 0.
\end{cases}
\]
The fundamental solution for this problem is given by
\begin{equation}\label{pa}
p^{(a)}(z,\zeta,t)  =(2t)^{-\frac{a+1}{2}}\left(\frac{z\zeta}{2t}\right)^{\frac{1-a}{2}}I_{\frac{a-1}{2}}\left(\frac{z\zeta}{2t}\right)e^{-\frac{z^2+\zeta^2}{4t}},
\end{equation}
where we have denoted by $I_\nu$ the modified Bessel function of the first kind. Formula \eqref{pa} is well-known in probability. For an explicit derivation based on purely analytical tools we refer the reader to  \cite[Section 22]{G18} or also \cite[Section 6]{EM}. We note that for every $z>0$ and $t>0$ one has
\begin{equation}\label{Patone}
\int_0^\infty p^{(a)}(z,\zeta,t) \zeta^a d\zeta = 1,
\end{equation}
see \cite[Proposition 2.3]{G18c}. Also, from \cite[Proposition 2.4]{G18c} we have
for every $z, \zeta>0$ and every $0<s, t<\infty$
\begin{equation}\label{ckpa}
p^{(a)}(z,\zeta,t) = \int_0^\infty p^{(a)}(z,\eta,t) p^{(a)}(\eta,\zeta,s) \eta^a d\eta.
\end{equation}

Using \eqref{pa} we now obtain the following result, whose verification is classical.

\begin{proposition}\label{P:Ga}
The \emph{Neumann fundamental solution} for the operator $\K_a$ in \eqref{Ka} with singularity at a point $(Y,\tau,\zeta)\in \RNu \times (0,\infty)$, is given by
\[
\mathscr G^{(a)}(X,t,z;Y,\tau,\zeta) = p(X,Y,t-\tau) p^{(a)}(z,\zeta,t-\tau),
\]
where $p(X,Y,t)$ is H\"ormander's fundamental solution of $\K$ in \eqref{PtKt} above.
\end{proposition}

By \cite[Remark 22.27]{G18} we see that if the pole of $\mathscr G^{(a)}$ is on the thin manifold $\RNu\times \{0\}$, and in particular at $(Y,0,0)$, then we have
$$\mathscr G^{(a)}(X,t,z;Y,0,0)  = \frac{1}{2^{a} \G(\frac{a+1}2)} t^{-\frac{a+1}2} e^{-\frac{z^2}{4t}} p(X,Y,t).$$

We note the following two basic properties of $\mathscr G^{(a)}$.

\begin{proposition}\label{P:2p}
For every $X\in \RN$, $z>0$ and $t>0$ one has
\[
\int_{\Rnp}  \mathscr G^{(a)}(X,t,z;Y,0,\zeta) \zeta^a dY d\zeta = 1.
\]
Furthermore, for $X, Y\in \RN$, $z, \zeta\ge 0$ and $t, s >0$, one has
\[
\mathscr G^{(a)}(X,t+s,z;Y,0,\zeta) = \int_{\Rnp}  \mathscr G^{(a)}(X,t,z;Z,0,\eta) \mathscr G^{(a)}(Z,s,\eta;Y,0,\zeta) \eta^a dZ d\eta.
\]
\end{proposition}

\begin{proof}
The proof of the first claim immediately follows from Tonelli's theorem, (i) in Lemma \ref{L:Pt} and from \eqref{Patone} above. To establish the second claim, we argue as follows. Tonelli's theorem again gives
\begin{align*}
& \int_{\Rnp}  \mathscr G^{(a)}(X,t,z;Z,0,\eta) \mathscr G^{(a)}(Z,s,\eta;Y,0,\zeta) \eta^a dZ d\eta
\\
& = \int_{\Rnp}  p(X,Z,t) p^{(a)}(z,\eta,t) p(Z,Y,s) p^{(a)}(\eta,\zeta,s) \eta^a dZ d\eta
\\
& = \int_{\RN} p(X,Z,t) p(Z,Y,s) dZ \int_0^\infty p^{(a)}(z,\eta,t) p^{(a)}(\eta,\zeta,s) \eta^a  d\eta
\\
& = p(X,Y,t+s) p^{(a)}(z,\zeta,t+s) = \mathscr G^{(a)}(X,t+s,z;Y,0,\zeta),
\end{align*}
where in the second to the last equality we have used Lemma \ref{L:Pt} $(v)$ and \eqref{ckpa} above.
\end{proof}
We refer the interested reader to the recent results in \cite[Section 5]{GTfi} for sharp pointwise estimates for $\mathscr G^{(a)}$ and the associated extension semigroup.

\begin{definition}\label{D:PaK}
We define the \emph{Poisson kernel} for the operator $\K_a$ as the function in $C^\infty(\RNu\times (0,\infty))$ given by
\begin{align*}
P^{(a)}_z(X,Y,t) & \overset{def}{=} - z^{-a} \p_z \mathscr G^{(-a)}(X,t,z;Y,0,0)
\\
& = \frac{1}{2^{1-a} \G(\frac{1-a}{2})} \frac{z^{1-a}}{t^{\frac{3-a}{2}}} e^{-\frac{z^2}{4t}}  p(X,Y,t).
\end{align*}
\end{definition}
We emphasize that, since $a\in (-1,1)$, we have $\frac{3-a}{2}>1$. 
The next result expresses a first basic property of the kernel $P^{(a)}_z(X,Y,t)$. 

\begin{proposition}\label{P:Pa1}
For every $(X,z)\in \RN \times \R^+$ one has
\[
\int_0^\infty \int_{\RN} P^{(a)}_z(X,Y,t) dY dt = 1.
\]
\end{proposition}

\begin{proof}
By Definition \ref{D:PaK} and Tonelli's theorem we have
\[
\int_0^\infty \int_{\RN} P^{(a)}_z(X,Y,t) dY dt = \frac{1}{2^{1-a} \G(\frac{1-a}{2})} \int_0^\infty \frac{z^{1-a}}{t^{\frac{3-a}{2}}} e^{-\frac{z^2}{4t}} dt \int_{\RN} p(X,Y,t) dY.
\]
The desired conclusion now follows from (i) in Lemma \ref{L:Pt} and from the observation that for every $z>0$ one has
\begin{equation}\label{ga1}
\frac{1}{2^{1-a} \G(\frac{1-a}{2})} \int_0^\infty \frac{z^{1-a}}{t^{\frac{3-a}{2}}} e^{-\frac{z^2}{4t}} dt = 1.
\end{equation}

\end{proof}

Another crucial property of $P^{(a)}_z(X,Y,t)$ is that it satisfies the partial differential equation $\K_a P^{(a)}_z(X,Y,t) = 0$, where $\K_a$ is the extension operator in \eqref{Ka}.

\begin{proposition}\label{P:pdePa}
Fix $(Y,0,0)\in \RNu\times \{0\}$. Then, in every $(X,t,z)\in \RNu\times (0,\infty)$ with $t>0$, one has
\[
z^{-a} \K_a P^{(a)}_z(X,Y,t)  = \K P^{(a)}_z(X,Y,t) + \Ba P^{(a)}_z(X,Y,t) = 0.
\]
\end{proposition}

\begin{proof}
For ease of computation let us denote
\begin{equation}\label{ga}
g^{(a)}(z,t) = \frac{1}{2^{1-a} \G(\frac{1-a}{2})} \frac{z^{1-a}}{t^{\frac{3-a}{2}}} e^{-\frac{z^2}{4t}},
\end{equation}
so that 
\begin{equation}\label{prodPa}
P^{(a)}_z(X,Y,t) = g^{(a)}(z,t)  p(X,Y,t).
\end{equation}
Keeping in mind that $\K = \mathscr A - \p_t$, we have
\begin{align*}
& \K P^{(a)}_z(X,Y,t) + \Ba P^{(a)}_z(X,Y,t) = g^{(a)}(z,t) \K p(X,Y,t) - p(X,Y,t) \p_t g^{(a)}(z,t)
\\
& + p(X,Y,t) \Ba g^{(a)}(z,t). 
\notag
\end{align*}
Since $\K p(X,Y,t) = 0$, we infer
\begin{align*}
& \K P^{(a)}_z(X,Y,t) + \Ba P^{(a)}_z(X,Y,t) = p(X,Y,t) \left(\Ba g^{(a)}(z,t) - \p_t g^{(a)}(z,t)\right).
\end{align*}
A computation now gives
\begin{equation}\label{eqga}
\Ba g^{(a)}(z,t) = \left(\frac{z^2}{4t^2} - \frac{3-a}{2t}\right) g^{(a)}(z,t) = \p_t g^{(a)}(z,t).
\end{equation}
We infer that $\Ba g^{(a)}(z,t) - \p_t g^{(a)}(z,t) = 0$, thus reaching the desired conclusion.

\end{proof}

We finally establish a lemma that will prove critical in the proof of Proposition \ref{P:PKextsolution} below.

\begin{lemma}\label{L:limits}
For every $X, Y\in \RN$, and any $z>0$, we have
\[ 
P^{(a)}_z(X,Y,\infty) \overset{def}{=}  \underset{t\to \infty}{\lim} P^{(a)}_z(X,Y,t) = 0, \ \ \text{and} \ \ \  P^{(a)}_z(X,Y,0) \overset{def}{=} \underset{t\to 0^+}{\lim} P^{(a)}_z(X,Y,t) = 0.
\]
\end{lemma}

\begin{proof}
We begin by observing that, by the definition of $K(t)$ in \eqref{Kt}, we have the monotonicity of $t\mapsto tK(t)$ (in the sense of matrices). This implies that, if we fix arbitrarily a number $t_0>0$, then by \eqref{PtKt} we have for every $t\geq t_0$ and
for all $X,Y\in\RN$
\[
0< p(X,Y,t)\leq \frac{\left(  4\pi\right)  ^{-N/2}}{\sqrt{\det \left(t_0K\left(  t_0\right)\right)  }}.
\]
Since on the other hand it is obvious from \eqref{ga} that for every $z>0$ we have $\underset{t\to \infty}{\lim} g^{(a)}(z,t) = 0$, then the conclusion regarding $P^{(a)}_z(X,Y,\infty)$ follows immediately by \eqref{prodPa}.\\
Concerning the behavior near $t=0$, we start noticing that, for every $X,Y$, and $t$, by the expression in \eqref{pcomeG} we easily have 
\[
0\leq p(X,Y,t)\leq \frac{\left(  4\pi\right)  ^{-N/2}}{\sqrt{\det C\left(  t\right)  }}
e^{ - t\, \mathrm{tr} B}.
\]
Furthermore, it can be seen from its definition that the matrix $C(t)$ (and thus $\det C\left(  t\right)$) behaves polynomially at $t=0$. We can in fact write, as $t\rightarrow 0^+$, $C(t)=tQ-\frac{1}{2}t^2\left(BQ+QB^\star\right)+o(t^2)$. More precisely, it is proved in \cite[equation (3.14) and Proposition 2.3]{LP} that $\det C\left(  t\right)$ is asymptotic to $t^{D_0}$ as $t\rightarrow 0^+$, where $D_0$ is the homogeneous dimension of a suitable homogeneous operator associated with $\K$. Hence, since $g^{(a)}(z,t)$ tends to $0$ exponentially for every $z>0$, we can conclude the proof by using again \eqref{prodPa}.

\end{proof}


\section{Solving the extension problem for $(-\K)^s$}\label{S:epKs}

In this section we solve the extension problem \eqref{epKa}. Using the Poisson kernel  $P^{(a)}_z(X,Y,t) $ we define an explicit solution formula, and prove that the latter does actually solve the problem \eqref{epKa}. 
The following theorem contains one of the main results of the present paper.

\begin{theorem}\label{T:epKinfty}
Given $0<s<1$, let $a = 1-2s$. Let $\K$ be given as in \eqref{K0}, with the assumption $K(t)>0$ for $t>0$  in force. Let $u\in \S$ and consider the function defined by the equation
\begin{align}\label{UKa}
U(X,t,z) & = \int_{-\infty}^t \int_{\RN} P^{(a)}_z(X,Y,t-\tau) u(Y,\tau) dY d\tau 
\\
& = \int_0^\infty \int_{\RN} P^{(a)}_z(X,Y,\tau) u(Y,t-\tau) dY d\tau .
\notag
\end{align}
Then, $U\in C^\infty(\RNu \times (0,\infty))$, and $U$ solves the extension problem in $L^\infty(\RNu)$, in the sense that we have $\K_a U = 0$ in $\RNu\times (0,\infty)$, and moreover
\begin{equation}\label{convKainfty}
\underset{z\to 0^+}{\lim} ||U(\cdot,\cdot,z) - u||_{L^\infty(\RNu)} = 0.
\end{equation}
Furthermore, we also have in $L^\infty(\RNu)$
\begin{equation}\label{nconvKainfty}
- \frac{2^{-a} \Gamma\left(\frac{1-a}2\right)}{\Gamma\left(\frac{1+a}2\right)}  \underset{z\to 0^+}{\lim} \paa U(\cdot,\cdot,z) = (-\K)^s u.
\end{equation}
\end{theorem}

\begin{proof}
We first prove that $U\in C^\infty(\RNu \times (0,\infty))$. With $(X,t,z)\in \RNu \times (0,\infty)$ fixed, we want to differentiate under the integral sign around $(X,t,z)$ by using the second equality in \eqref{UKa}. From \eqref{prodPa} and the Gaussian character of $g^{(a)}$ in \eqref{ga}, there is no problem in differentiating with respect to the $z$-variable. Moreover, since $u\in\S$ and it has bounded $t$-derivatives, also $\de_t U$ can be performed easily. The problems might arise when we differentiate with respect to $X$, and in particular concerning the behavior in $\tau$ (for both $\tau\rightarrow 0^+$ and $\tau\rightarrow \infty$) of
$$\nabla_X\left( P^{(a)}_z(X,Y,\tau) u(Y,t-\tau)\right)=g^{(a)}(z,\tau)u(Y,t-\tau)\nabla_X p(X,Y,\tau).$$
A direct computation shows that
\begin{equation}\label{Xderp}
\nabla_X p(X,Y,\tau)=-\frac{1}{2}C^{-1}(\tau)\left(X-e^{-\tau B}Y\right)p(X,Y,\tau).
\end{equation}
On one side, for small $\tau$, we can bound
$$|\nabla_X p(X,Y,\tau)|\leq c(N,B) \left\|C^{-1}(\tau)\right\|\left(|X|+|Y|\right)p(X,Y,\tau),$$
and we can use the fact that, as we have mentioned in  $C(\tau)$ behaves like a polynomial for small $\tau$ (see, e.g., \cite[Lemma 3.3]{LP} for a precise behavior). Hence, thanks to the Gaussian behavior of $g^{(a)}(z,\tau)p(X,Y,\tau)$ (we recall that $z>0$ and $u\in\S$), we can find a uniform bound for $\left|\nabla_X\left( P^{(a)}_z(X,Y,\tau) u(Y,t-\tau)\right)\right|$ which is in $L^1(\RN\times (0,1))$. We now have to consider the behavior for large values of $\tau$. We notice that we can write
\begin{eqnarray*}
&&|C^{-1}(\tau)\left(X-e^{-\tau B}Y\right)|^2\leq \left\|C^{-\frac{1}{2}}(\tau)\right\|^2\left\langle C^{-1}(\tau)\left(X-e^{-\tau B}Y\right),\left(X-e^{-\tau B}Y\right) \right\rangle\\
&\leq& \left\|C^{-\frac{1}{2}}(\tau)\right\|^2\left(\left\langle C^{-1}(\tau)X,X \right\rangle + \left\langle e^{-\tau B^\star}C^{-1}(\tau)e^{-\tau B}Y,Y \right\rangle+\right.\\
&+&\left.2\left\langle C^{-1}(\tau)X,X \right\rangle^{\frac{1}{2}}\left\langle e^{-\tau B^\star}C^{-1}(\tau)e^{-\tau B}Y,Y \right\rangle^{\frac{1}{2}}\right).
\end{eqnarray*}
Furthermore, from $C(t) = \int_0^t e^{-sB} Q e^{-sB^\star} ds$ it is obvious that $C(\tau)\geq C(\tau_0)$ for all $\tau\geq \tau_0>0$, and \eqref{KC} gives $e^{\tau B}C(\tau)e^{\tau B^\star}=\tau K(\tau)\geq \tau_0 K(\tau_0)= e^{\tau_0 B}C(\tau_0)e^{\tau_0 B\star}$. Fixing $\tau_0=1$, we then infer that for all $\tau\geq 1$,
\begin{eqnarray*}
&&|C^{-1}(\tau)\left(X-e^{-\tau B}Y\right)|^2\leq \left\|C^{-\frac{1}{2}}(1)\right\|^2\left(\left\langle C^{-1}(1)X,X \right\rangle + \left\langle e^{-B^\star}C^{-1}(1)e^{-B}Y,Y \right\rangle+\right.\\
&+&\left.2\left\langle C^{-1}(1)X,X \right\rangle^{\frac{1}{2}}\left\langle e^{-B^\star}C^{-1}(1)e^{- B}Y,Y \right\rangle^{\frac{1}{2}}\right).
\end{eqnarray*}
This estimate, together with \eqref{Xderp} and the behaviour of $g^{(a)}(z,\tau)$ for large values of $\tau$, allows to find a uniform bound for $\left|\nabla_X\left( P^{(a)}_z(X,Y,\tau) u(Y,t-\tau)\right)\right|$ which is in $L^1(\RN\times (1,+\infty))$. This proves that we can differentiate (at least one time) $U$ under the integral sign around any $(X,t,z)$. We can argue in the same way for derivatives of arbitrary order. Therefore, $U\in C^\infty(\RNu \times (0,\infty))$ and, by Proposition \ref{P:pdePa}, we can say that
$$ \K_aU(X,t,z)=\int_0^\infty \int_{\RN} \K_a P^{(a)}_z(X,Y,\tau) u(Y,t-\tau) dY d\tau =0$$
for all $(X,t,z)\in\RNu\times (0,+\infty)$. As a second step we show that \eqref{convKainfty} holds. To reach this goal we make the  observation that $U$ can be written in the following form using the semigroup $P^\K_\tau$
\begin{equation}\label{rem}
U(X,t,z) = \frac{1}{2^{1-a} \G(\frac{1-a}{2})} z^{1-a} \int_0^\infty \frac{1}{\tau^{\frac{3-a}{2}}} e^{-\frac{z^2}{4\tau}} P^\K_\tau u(X,t) d\tau.
\end{equation}
To recognize the validity of \eqref{rem} we use the second equality in \eqref{UKa} and \eqref{prodPa} to find
\begin{align*}
U(X,t,z) & = \int_0^\infty \int_{\RN} P^{(a)}_z(X,Y,\tau) u(Y,t-\tau) dY d\tau 
\\
&= \int_0^\infty g^{(a)}(z,\tau) \left(\int_{\RN} p(X,Y,\tau) u(Y,t-\tau) dY\right) d\tau
\\
& = \int_0^\infty g^{(a)}(z,\tau) P^\K_\tau u(X,t) d\tau,
\end{align*}
where in the last equality we have used \eqref{eks} above. Keeping \eqref{ga} in mind, we have proved \eqref{rem}.

In view of \eqref{ga1} we now obtain from \eqref{rem} that we can also write
\begin{align}\label{beauty00}
& U(X,t,z) - u(X,t) 
\\
& = \frac{1}{2^{1-a} \G(\frac{1-a}{2})} z^{1-a} \int_0^\infty \frac{1}{\tau^{\frac{3-a}{2}}} e^{-\frac{z^2}{4\tau}} \left[P^\K_\tau u(X,t) - u(X,t)\right] d\tau.
\notag
\end{align}
Using the representation \eqref{beauty00} we can now write
\begin{align*}
& ||U(\cdot,\cdot,z) - u||_{L^\infty(\RNu)} 
\\
&  \le \frac{1}{2^{1-a} \G(\frac{1-a}{2})} z^{1-a} \int_0^1 \frac{1}{\tau^{\frac{3-a}{2}}} e^{-\frac{z^2}{4\tau}} \left\|P^\K_\tau u - u\right\|_{L^\infty(\RNu)}  d\tau
\\
& + \frac{1}{2^{1-a} \G(\frac{1-a}{2})} z^{1-a} \int_1^\infty \frac{1}{\tau^{\frac{3-a}{2}}} e^{-\frac{z^2}{4\tau}} \left\|P^\K_\tau u - u\right\|_{L^\infty(\RNu)}  d\tau.
\end{align*}
In the second integral we use the contractivity of $P^\K_\tau$ on $L^\infty(\RNu)$ (Lemma \ref{L:PtK}) to bound
\[
\frac{1}{\tau^{\frac{3-a}{2}}} e^{-\frac{z^2}{4\tau}} \left\|P^\K_\tau u - u\right\|_{L^\infty(\RNu)} \le 2 \left\|u\right\|_{L^\infty(\RNu)} \frac{1}{\tau^{\frac{3-a}{2}}}\in L^1(1,\infty),
\]
since  $\frac{3-a}{2}>1$. In the first integral, instead, we need to crucially use the rate in Lemma \ref{P:Kdiff}
\[
\left\|P_\tau^\K u - u\right\|_{L^\infty(\RNu)} = O(\tau),
\]
to estimate
\begin{align*}
& \int_0^1 \frac{1}{\tau^{\frac{3-a}{2}}} e^{-\frac{z^2}{4\tau}} \left\|P^\K_\tau u - u\right\|_{L^\infty(\RNu)}  d\tau 
 \le C \int_0^1 \frac{1}{\tau^{\frac{1-a}{2}}}  d\tau < \infty, 
\end{align*}
since $0< \frac{1-a}{2}<1$. 
In conclusion, the right-hand side in \eqref{beauty00} goes to $0$ in $L^\infty(\RNu)$ norm with $z^{1-a}$, and since $1-a>0$, we have demonstrated \eqref{convKainfty}.

In order to complete the proof we are left with establishing \eqref{nconvKainfty}. The proof of this hinges again on the representation formula \eqref{beauty00}. Differentiating it, we find
\begin{align}\label{wow}
& - \frac{2^{-a} \Gamma\left(\frac{1-a}2\right)}{\Gamma\left(\frac{1+a}2\right)}  \paa U(X,t,z) 
\\
& = - \frac{1-a}{2\Gamma\left(\frac{1+a}2\right)}  \int_0^\infty \frac{1}{\tau^{\frac{3-a}{2}}} e^{-\frac{z^2}{4\tau}} \left[P^\K_\tau u(X,t) - u(X,t)\right] d\tau 
\notag
\\
& +  \frac{1}{4\Gamma\left(\frac{1+a}2\right)} z^2 \int_0^\infty \frac{1}{\tau^{\frac{3-a}{2}}} e^{-\frac{z^2}{4\tau}} \left[P^\K_\tau u(X,t) - u(X,t)\right] \frac{d\tau}\tau.
\notag
\end{align}
On the other hand, keeping in mind that $a = 1-2s$, we can rewrite the definition \eqref{Ks} as follows
\begin{equation}\label{newclothes}
(-\K)^s u(X,t) = - \frac{1-a}{2\Gamma\left(\frac{1+a}2\right)}  \int_0^\infty \frac{1}{\tau^{\frac{3-a}{2}}} \left[P^\K_\tau u(X,t) - u(X,t)\right] d\tau.
\end{equation}
Subtracting \eqref{newclothes} from \eqref{wow} we thus find
\begin{align*}
& \left\|- \frac{2^{-a} \Gamma\left(\frac{1-a}2\right)}{\Gamma\left(\frac{1+a}2\right)}  \paa U(\cdot,\cdot,z) - (-\K)^s u\right\|_{L^\infty(\RNu)} 
\\
& \le \frac{1-a}{2\Gamma\left(\frac{1+a}2\right)}  \int_0^\infty \frac{1}{\tau^{\frac{3-a}{2}}} \left|e^{-\frac{z^2}{4\tau}} - 1\right| \left\|P^\K_\tau u - u\right\|_{L^\infty(\RNu)}  d\tau 
\\
& + \frac{z^2}{4\Gamma\left(\frac{1+a}2\right)}  \int_0^\infty \frac{1}{\tau^{\frac{3-a}{2}}} e^{-\frac{z^2}{4\tau}} \left\|P^\K_\tau u - u\right\|_{L^\infty(\RNu)}  \frac{d\tau}\tau 
\\
& = I(z) + II(z).
\end{align*}
To complete the proof of the theorem it suffices to show that both $I(z), II(z) \longrightarrow 0$ as $z\to 0^+$. We handle $II(z)$ as follows
\begin{align*}
& II(z) \cong z^2 \int_0^1 \frac{1}{\tau^{\frac{1-a}{2}}} e^{-\frac{z^2}{4\tau}}  \frac{d\tau}\tau
 + z^2 \int_1^\infty \frac{1}{\tau^{\frac{3-a}{2}}} \frac{d\tau}\tau
 \\
 & = O(z^{1+a}) \ \longrightarrow\ 0\ \ \ \ \ \text{since}\ a\in (-1,1).
\end{align*}
For $I(z)$ we consider the integrand
\[
0\le g_z(\tau) \overset{def}{=} \frac{1}{\tau^{\frac{3-a}{2}}} \left|e^{-\frac{z^2}{4\tau}} - 1\right| \left\|P^\K_\tau u - u\right\|_{L^\infty(\RNu)}, \ \ \ \ \ \ 0<\tau <\infty.
\]
We clearly have $g_z(\tau) \to 0$ as $z\to 0^+$ for every $\tau>0$. On the other hand, there exists an absolute constant $C>0$ and a function $g\in L^1(0,\infty)$ such that $0\le g_z(\tau) \le C g(\tau)$ for every $\tau >0$. 
In fact, using Lemmas \ref{L:PtK} and \ref{P:Kdiff}
it is not difficult to convince oneself that we can take
\[
g(\tau) = \begin{cases}
\frac{1}{\tau^{\frac{1-a}{2}}} \ \ \ \ \ \ 0<\tau \le 1,
\\
\frac{1}{\tau^{\frac{3-a}{2}}}\ \ \ \ \ \ \ 1<\tau<\infty.
\end{cases}
\]
By Lebesgue dominated convergence we conclude that $I(z)\to 0$ as $z\to 0^+$.
\end{proof}

We can now state the second main result in this paper.

\begin{theorem}\label{T:epKL2}
Suppose that \eqref{trace} holds. 
 Let $u\in \S$ and consider the function $U$ defined by \eqref{UKa} above.
Then, $U\in C^\infty(\RNu \times (0,\infty))$, and $U$ solves the extension problem in $L^p(\RNu)$ for any $1\leq p<\infty$. In the sense that we have $\K_a U = 0$ in $\RNu\times (0,\infty)$, and moreover
\begin{equation}\label{convKaL2}
\underset{z\to 0^+}{\lim} ||U(\cdot,\cdot,z) - u||_{L^p(\RNu)} = 0.
\end{equation}
Furthermore, we also have in $L^p(\RNu)$
\begin{equation}\label{nconvKaL2}
- \frac{2^{-a} \Gamma\left(\frac{1-a}2\right)}{\Gamma\left(\frac{1+a}2\right)}  \underset{z\to 0^+}{\lim} \paa U(\cdot,\cdot,z) = (-\K)^s u.
\end{equation}
\end{theorem}

\begin{proof}
We begin by observing that, in view of Remark \ref{R:L2K}, the assumption \eqref{trace} guarantees that $(-\K)^s u\in L^p(\RNu)$. Next, since the first part of the theorem has already been established in the proof of Theorem \ref{T:epKinfty} we only need to show that \eqref{convKaL2} and \eqref{nconvKaL2} hold. Now, the proof of these facts proceeds exactly as in the proof of \eqref{convKainfty} and \eqref{nconvKainfty}, except that we must replace $L^\infty$ norms with $L^p$ ones, which we can do since by (v) in Lemma \ref{L:PtK} we know that the semigroup $P^\K_\tau$ is contractive in $L^p(\RNu)$. For the integrals near zero, say on the interval $(0,1)$, we use the crucial convergence rate in Lemma \ref{P:Kdiff}, and everything proceeds as in the proof of Theorem \ref{T:epKinfty}. 

\end{proof}


\section{The extension problem for the nonlocal operator $(-\mathscr A)^s$}\label{S:fracA}

In this last section we use the results of Section \ref{S:epKs} and Bochner's subordination to solve the extension problem for the fractional powers \eqref{As} of the hypoelliptic operators $\mathscr A$ which constitute the ``diffusive" part of the H\"ormander operators $\K$ in \eqref{K0}. Since once the properties of the relevant Poisson kernel are established the details are completely analogous to those in Theorems \ref{T:epKinfty} and \ref{T:epKL2}, we will skip them altogether. 

We consider the space $\Rnp = \R^N \times (0,\infty)$, and use the letters $(X,z), (Y,\zeta)$, etc. to indicate generic points in such space. For any number $a\in (-1,1)$ we now consider the following partial differential operator in $\Rnp$
\begin{equation}\label{extA}
\sA_a \overset{def}{=} z^a \left(\sA + \Ba\right).
\end{equation}
Again in analogy with \cite{CS}, when $a = 1-2s$ we call the operator $\sA_a$ in \eqref{extA} the \emph{extension operator} for $(-\sA)^s$ in \eqref{As}. We now introduce the following.

\begin{definition}\label{D:extpoisson}
Given any $a\in (-1,1)$, we define the \emph{Poisson kernel} for the operator $\sA_a$ in \eqref{extA} above as
\begin{equation}\label{slPK}
\Po^{(a)}(X,Y,z) =  \int_0^\infty P^{(a)}_z(X,Y,t)  dt,\qquad X,Y\in\RN,\,\, z>0,
\end{equation}
where the function $P^{(a)}_z(X,Y,t) $ is as in Definition \ref{D:PaK}.
\end{definition} 

A first basic property of the kernel $\Po^{(a)}(X,Y,z) $ is expressed by the next result.

\begin{proposition}\label{P:Ka1}
For every $X\in \R^N$ and $z>0$, one has
\[
\int_{\R^N} \Po^{(a)}(X,Y,z) dY = 1.
\]
\end{proposition}

\begin{proof}
Using \eqref{slPK} and Tonelli's theorem we find for every $X\in \R^N$
\begin{align*}
& \int_{\R^N} \Po^{(a)}(X,Y,z) dY = \int_{\R^N} \int_0^\infty P^{(a)}_z(X,Y,t)  dt dY = \int_0^\infty \int_{\R^N}  P^{(a)}_z(X,Y,t)   dY dt = 1,
\end{align*}
where in the last equality we have used Proposition \ref{P:Pa1}. 

\end{proof}

We now show that the kernel  $\Po^{(a)}(X,Y,z)$ is a solution of the extension operator $\sA_a$ in \eqref{extA} above. 

\begin{proposition}\label{P:PKextsolution} 
Fix $Y\in \R^N$. The function $(X,z)\to \Po^{(a)}(X,Y,z)$ belongs to $C^\infty(\RN\times (0,\infty))$. Furthermore, for every $X\not= Y$ and $z>0$ one has
\[
\sA_{a} \Po^{(a)}(X,Y,z) = 0.
\]
\end{proposition}

\begin{proof}
We show that we can differentiate under the integral sign in the definition \eqref{slPK} and prove that $(X,z)\to \Po^{(a)}(X,Y,z)$ belongs to $C^\infty(\RN \times (0,\infty))$. To do this, we have on one side that the required bound on the $z$-derivatives is straightforward (since $z>0$). On the other side, we need to be careful when we differentiate with respect to the $X$-variables. However, this can be done by arguing as in the proof of Theorem \ref{T:epKinfty}, where we establish the right bounds of $\nabla_X P_z^{(a)}(X,Y,t)$ respectively for small values and large values of $t$. In this way we accomplish the first part of the statement. 
Furthermore, using \eqref{extA} we find for any $z>0$ and $X\not = Y$
\[
z^{-a} \sA_{a} \Po^{(a)}(\cdot,Y,z) = \sA \Po^{(a)}(\cdot,Y,z) + \Ba \Po^{(a)}(\cdot,Y,z).
\]
To compute the quantities in the r.h.s., we differentiate under the integral sign obtaining 
\begin{align}\label{KaLL}
& z^{-a} \sA_{a} \Po^{(a)}(X,Y,z) = \int_0^\infty g^{(a)}(z,t) \sA p(X,Y,t) dt
+ \int_0^\infty p(X,Y,t) \Ba g^{(a)}(z,t) dt.
\end{align}
To compute the first integral in the right-hand side of \eqref{KaLL} we now use the equation satisfied by $p(X,Y,t)$, $\mathscr A p(X,Y,t) = \p_t p(X,Y,t)$. This gives for every $X\not=Y$ and $t>0$,
\begin{align}\label{KaLL2}
& \int_0^\infty g^{(a)}(z,t) \sA p(X,Y,t) dt = \int_0^\infty g^{(a)}(z,t) \p_t p(X,Y,t) dt
\\
& \ \ \ \ \ \ \ \text{(integrating by parts)}
\notag\\
& =  P^{(a)}_z(X,Y,\infty) - P^{(a)}_z(X,Y,0) - \int_0^\infty \p_t g^{(a)}(z,t) \p_t p(X,Y,t) dt
\notag\\
& = - \int_0^\infty p(X,Y,t) \Ba g^{(a)}(z,t)  dt,
\notag
\end{align}
where in the last equality we have used the crucial Lemma \ref{L:limits} and \eqref{eqga}. Substituting \eqref{KaLL2} into \eqref{KaLL} we reach the desired conclusion.

\end{proof}
 
In closing, we solve the extension problem for the operator $(-\sA)^s$.

\begin{definition}\label{D:ep}
For $0<s<1$, let $a= 1-2s$. The extension problem in $\Rnp$ for the nonlocal operator $(-\sA)^s$, is, for a given $\vf\in\So$, the following:
\begin{equation}\label{ep}
\begin{cases}
\sA_a U = 0,\ \ \ \ \ \ \ \ \ \ \ \ \ \ \ \text{in}\ \Rnp,
\\
U(X,0) = \vf(X)\ \ \ \ \ \ \ X\in \R^N.
\end{cases}
\end{equation}
\end{definition}

Our final result is the counterpart of Theorem \ref{T:epKinfty}. Since the details are completely analogous we omit them altogether. 

\begin{theorem}\label{T:extLinfty}
Given $\vf\in \So$ consider the function $U$ defined by 
\begin{equation}\label{pos}
U(X,z) = \int_{\R^{N}} \Po^{(a)}(X,Y,z) \vf(Y) dY.
\end{equation}
One has $U\in C^\infty(\RN\times (0,\infty))$ and solves the extension problem \eqref{ep}. By this we mean that $\sA_a U = 0$ in $\Rnp$, and we have $\underset{z\to 0^+}{\lim} U(\cdot;z) = \vf$ in $L^\infty$.
Moreover, we also have in $L^\infty$
\begin{equation}\label{nconvAinfty}
- \frac{2^{-a} \Gamma\left(\frac{1-a}2\right)}{\Gamma\left(\frac{1+a}2\right)}  \underset{z\to 0^+}{\lim} \paa U(\cdot,z) = (-\mathscr A)^s u.
\end{equation}
If furthermore the hypothesis \eqref{trace} is satisfied, then the convergence is also in $L^p$ for any $1\leq p<\infty$.
\end{theorem}

In closing we mention that when $\sA$ is a nondegenerate Ornstein-Uhlenbeck operator $\Delta + <BX,\nabla>$, then some properties of $(-\mathscr{A})^{1/2}$ were obtained by Priola in \cite{Pri} in his study of the Dirichlet problem in half-spaces.




\bibliographystyle{amsplain}

\end{document}